\documentclass[11pt]{amsart}
\usepackage{amsmath,amssymb,enumerate}
\usepackage{latexsym}
\usepackage{amsfonts}
\usepackage {amsbsy}
\usepackage {amsmath}
\usepackage{amssymb}
\usepackage{enumerate}

\newtheorem{theorem}{Theorem}
\newtheorem{lemma}[theorem]{Lemma}
\newtheorem{corollary}[theorem]{Corollary}

\newtheorem{proposition}[theorem]{Proposition}
\newtheorem{definition}[theorem]{Definition}
\newtheorem{example}[theorem]{Example}

\newtheorem{Remark}[theorem]{Remark}
\numberwithin{theorem}{section}
\numberwithin{equation}{section}

\newcommand{\B}{\mathbb{B}}
\newcommand{\R}{\mathbb{R}}
\newcommand{\C}{\mathbb{C}}
\newcommand{\Cc}{\mathcal{C}}

\newcommand{\al}{\alpha}
\newcommand{\ep}{\epsilon}
\newcommand{\Om}{\Omega}
\newcommand{\Omf}{\partial\Omega}
\newcommand{\Omb}{\bar{\Omega}}
\newcommand{\fami}{\mathcal{V}(\Om,\fii,f)}
\newcommand{\fii}{\varphi}
\newcommand{\D}{\Delta_H }
\newcommand{\U}{\mathtt{U}}

\begin{document}
\title[  H\"older Regularity for solutions to CMAE]{ H\"older regularity for solutions to \\  complex  Monge-Amp\`ere equations}

\author{Mohamad CHARABATI}

\date{\today}

\begin{abstract}
We consider the Dirichlet problem for the complex Monge-Amp\`ere equation in a bounded strongly hyperconvex Lipschitz domain in $\C^n$. We first give a sharp estimate on the modulus of continuity of the solution when the boundary data is continuous and the right hand side has a continuous density. Then we consider the case when the boundary value function is $C^{1,1}$ and the right hand side has a density in $L^p(\Omega)$ for some $p>1$ and prove the H\"older continuity of the solution.
\end{abstract}

\maketitle

\section{Introduction}

Let $\Om$ be a bounded  pseudoconvex   domain in $\C^n$. Given $\fii\in\Cc(\Omf) $ and $ 0 \leq  f\in L^1 (\Omega)$.
We consider the Dirichlet problem:
\begin{center}
$ Dir(\Om,\fii,f):$
$ \left\{\begin{array}{ll}
            u\in PSH(\Om)\cap\Cc(\Omb)& \\
           (dd^c u)^n=f \beta^n  & \text{  in  } \; \Om \\
            u=\fii  & \text{ on  }  \Omf \\
         \end{array} \right. $

\end{center}
where $ PSH(\Om)$ is the set of plurisubharmonic (psh) functions in $\Om$. Here we denote $ d= \partial + \bar{\partial}$ and $ d^c = i/4(\bar{\partial} - \partial)$ then $ dd^c=i/2 \partial \bar{\partial} $ and $ (dd^c.)^n $ stands for the complex  Monge-Amp\`ere operator.

If $ u \in \Cc^2(\Om) $ and is plurisubharmonic function, the complex Monge-Amp\`ere operator is given by 
 $$ (dd^c u)^n= det \left(\frac{\partial^2 u}{\partial z_j \partial \bar{z_k}} \right) \beta^n $$
 where $ \beta = i/2 \sum_{j=1}^{n} dz_j \wedge d\bar{z_j}$ be the standard K\"ahler form in $\C^n$.

In their seminal work, Bedford and Taylor proved that the complex Monge-Amp\`ere operator can be extended to the set of bounded  plurisubharmonic functions (see \cite{BT76}, \cite{BT82}).  Moreover, it is invariant under holomorphic change of coordinates.
We refer the reader to \cite{BT76}, \cite{De89}, \cite{Kl91}, \cite{Ko05} for more details on its properties.

\smallskip
This problem has been studied extensively in last decades by many authors. When $\Om$ is a bounded strongly pseudoconvex domain with smooth boundary, Bedford and Taylor had showed that $Dir(\Om, \fii, f)$ has a unique continuous solution $\U:=\U(\Om,\fii, f)$. Furthermore, it was proved in \cite{BT76} that $ \U \in Lip_\alpha (\Omb) $ when $\fii \in Lip_{2\al} (\Omf) $ and $ f^{1/n} \in Lip_\al (\Omb) $ ($ 0 < \al \leq 1$). In the non degenerate case i.e. $ 0 < f\in \Cc^\infty (\Omb)$ and $ \fii \in \Cc^\infty (\Omf)$,  Caffarelli, Kohn, Nirenberg and Spruck proved in \cite{CKNS85} that $ \U \in \Cc^\infty (\Omb) $. However a simple example of  Gamelin and Sibony shows that the solution is not, in general, better than $\Cc^{1,1} $-smooth when $ f \geq 0 $ and smooth (\cite{GS80}). Krylov proved  that if $\fii \in \Cc^{3,1} (\Omf) $ and $ f^{1/n} \in \Cc^{1,1} (\Omb) $, $f \geq 0$ then $\U \in \Cc^{1,1}(\Omb)$ (see \cite{Kr89}). 

For $B$-regular domains, Blocki \cite{Bl96} proved the existence of a continuous solution to the Dirichlet problem $Dir(\Om,\fii,f)$ when $f \in \Cc(\Omb)$.

For a strongly pseudoconvex domain with smooth boundary, Ko\l odziej demonstrated in \cite{Ko98} that $ Dir(\Om, \fii , f) $ still admit a unique continuous solution under the  milder assumption $f \in L^p (\Om) $, for $p>1$. Recently Guedj, Kolodziej and Zeriahi studied the H\"older continuity of the solution when $ 0 \leq f\in L^p (\Om) $, for some $ p>1$, is bounded near the boundary (see \cite{GKZ08}).

For the complex Monge-Amp\`ere equation on a compact K\"ahler manifold, H\"older continuity of the solution was proved earlier by 
Ko\l odziej \cite{Ko08} (see also \cite{DDGHKZ12}).

A viscosity approach to the complex Monge-Amp\`ere equation has been developed in \cite{EGZ11} and \cite{Wan12}.

\smallskip
In this paper, we consider the more general case where $\Om$ be a bounded  strongly hyperconvex  Lipschitz domain (the boundary does not need to be smooth) and $ f \in L^p (\Omega)$.

We will generalize the approach of Bedford and Taylor \cite{BT76} by showing an estimate for the modulus of continuity to the solution in terms of the modulus of continuity of the data.
\smallskip

\noindent {\bf Theorem A.}
{\it
Let $ \Om \subset \C^n $ be a bounded strongly hyperconvex Lipschitz domain, $ \fii \in \Cc(\Omf)$ and $ 0 \leq f \in \Cc(\Omb)$. Assume  that $ \omega_\fii  $ is the modulus of continuity of $\fii$ and $ \omega_{f^{1/n}} $ is the modulus of continuity of $ f^{1/n}$. Then the modulus of continuity of $\U$  has the following estimate
   $$ \omega_\U (t) \leq \eta  (1+\| f\|^{1/n}_{L^{\infty}(\Omb)} )  \max \{\omega_\fii(t^{1/2}), \omega_{f^{1/n}}(t), t^{1/2} \}  $$
where  $ \eta $ is a positive constant depending on $ \Om$.
}

Here we will use a new description of the solution given by Proposition \ref{solution} to get an optimal control for the modulus of continuity of this solution in a strongly hyperconvex Lipschitz domain. 

\medskip

For more general density $ f \in L^p(\Om)$ for some $ p>1$, it was shown in \cite{GKZ08} that the unique solution to $Dir(\Om, \fii,f)$ belongs to $ \Cc^{0,\al} (\Omb)$ for all $ \alpha < 2/(nq+1)$ when $ \fii \in \Cc^{1,1}(\Omf)$ and $ f \in L^p(\Om)$ be a bounded function near the boundary. Here we will improve this result and show the following theorem  
\smallskip

\noindent {\bf Theorem B.}
\textit{
Let $ \Om \Subset \C^n$ be a bounded strongly hyperconvex Lipschitz domain. Assume that $ \fii \in \Cc^{1,1}(\Omf)$ and $ f \in L^p(\Om)$ for some $p >1$. Then the unique solution $\U$ to $ Dir(\Om,\fii,f)$ is $ \alpha$-H\"older continuous on $\Omb$ for any $ 0< \alpha < 1/(nq+1)$ where $ 1/p+1/q=1$. Moreover, if  $p \geq 2$, then the solution $\U$ is $ \alpha$-H\"older continuous on $\Omb$ for any $ 0< \alpha < min \{ 1/2, 2/(nq+1) \}$.
}

\bigskip

\noindent {\bf Acknowledgements}.
I would like to express my deepest and sincere gratitude to my advisor, Professor Ahmed Zeriahi, for all his help and sacrificing his very valuable time for me.  I also would like to thank Hoang Chinh Lu for valuable discussions.  I wish to express my  acknowledgement to Professor Vincent
Guedj for  useful  discussions.

\bigskip

\section{Preliminaries} 

We recall that a hyperconvex domain is to be a domain in $\C^n$ admitting a bounded exhaustion function.

Let us define the class of hyperconvex domains which will be considered in this paper.

\begin{definition}\label{lip}
A bounded domain $ \Om \subset \C^n$ is called strongly hyperconvex Lipschitz (shortly SHL) domain if there exists a neighbourhood $\Om'$ of $\Omb$ and a  Lipschitz plurisubharmonic function $ \rho : \bar{\Om}' \to \R $ such that  
\begin{enumerate}
\item $ \rho < 0$ in $\Om$ and $\partial \Omega = \{  \rho = 0\}$,
\item there exists a constant $c>0$ such that $ dd^c \rho \geq c \beta $ in  $\Omega$ in the weak sense of currents.
\end{enumerate}
\end{definition}

\begin{example}\label{hyperconvex}
\
\begin{enumerate}
\item Let $\Om$ be a strictly convex domain that is there exists a Lipschitz defining function $ \rho $ such that $ \rho - c |z|^2$ is convex  for some $c>0$. It is clear that $\Om$ is strongly hyperconvex  Lipschitz domain.
\item A smooth strictly pseudoconvex bounded domain is a SHL  domain (see \cite{HL84}).
\item The nonempty finite intersection of  strictly pseudoconvex bounded domains with smooth boundary in $ \C^n$  is a bounded SHL  domain. In fact, it is sufficient to put $ \rho = max \{\rho_i \} $.
More generally a finite intersection of SHL domains is an SHL domain.
\item The domain $ \Om = \{ z = (z_1,\cdots,z_n) \in \C^n ; |z_1| + \cdots + |z_n| < 1 \}$ ($n \geq 2$) is a bounded strongly hyperconvex Lipschitz domain in $\C^n$ with non smooth boundary.
\item  The unit polydisc in $\C^n$ ($n \geq 2)$ is hyperconvex with Lipschitz boundary but it is not a strongly hyperconvex Lipschitz.
\end{enumerate}

\end{example}

\begin{Remark}
Kerzman and Rosay \cite{KR81} proved that in a hyperconvex domain there exists there exists an exhaustion function which is smooth and strictly plurisubharmonic. Furthermore, they proved that any bounded pseudoconvex domain with $\Cc^1$-boundary is hyperconvex domain. This result was extended by Demailly  \cite{De87} to bounded pseudoconvex domains with Lipschitz boundary. 
\end{Remark}
 
Let $\Om \subset \C^n$ be a bounded  domain. If $u\in PSH(\Om) $ then $dd^c u \geq 0$ in the sense of currents.  We define 
\begin{equation}\label{laplacian}
\D u:=  \sum\limits_{j,k=1}^{n} h_{j\bar{k}} \frac{\partial ^2u}{\partial z_k \partial{\bar{z_j}}}
\end{equation}
for every positive definite Hermitian matrix $ H=(h_{j\bar{k}})$. We can see $ \D u$ as a positive Radon measure in $\Om$.

The following lemma is elementary  and important for the sequel (see \cite{Gav77}). 

\begin{lemma}(\cite{Gav77}).\label{det}
Let  $ Q $  be a  $ n \times n $ nonnegative hermitian matrix. Then
\begin{center}
$ \left(det Q\right)^\frac{1}{n}=\inf \{tr(H.Q); H\in H_n^+ \; and \; det(H)=n^{-n} \} $
\end{center}
where $  H_n^+ $ denotes the set of all positive hermitian $ n \times n $ matrices.
\end{lemma}

\begin{example}
We calculate $ \D (|z|^2)$ for every matrix $ H\in H_n^+$ and $ det H = n^{-n} $.
$$ \D ( | z |^2) =  \sum\limits_{j,k=1}^{n} h_{j\bar{k}} . \delta_{k\bar{j}}= tr(H) $$
using the inequality of arithmetic and geometric means, we have :
$$ 1 = (det I)^{\frac{1}{n}} \leq tr(H) ,$$
hence $ \; \;  \D (| z |^2) \geq 1 $
for every matrix $ H \in H_n^+ $ and $ det(H) = n^{-n} $. 
\end{example}
 Using ideas from the theory of viscosity due to Eyssidieux, Guedj and Zeriahi \cite{EGZ11},  we can prove the following result.
\begin{proposition} \label{equi}
Let $ u \in PSH\cap L^{\infty} (\Om)$ and $ 0 \leq f \in \Cc(\Omb)$. Then the following conditions are equivalent:\\
1) $ \D u \geq f^{1/n} $ in the weak sense of distributions, for any $ H\in H_n^+ $ and $ det H=n^{-n} $. \\
2) $(dd^c u)^n \geq f \beta^n $  in the weak sense of currents in $\Om $.
\end{proposition}
This result is implicitly contained in \cite{EGZ11}, but we will give a complete proof for the convenience of the reader.
\begin{proof} 
First, suppose that $u \in \Cc^2(\Om)$,then by  Lemma  \ref{det} the following 
$$ \D u =  \sum\limits_{j,k=1}^{n} h^{j\bar{k}} \frac{\partial ^2u}{\partial z_j \partial{\bar{z_k}}}   \geq f^{1/n} , \forall H \in H_n^+, det(H) = n^{-n} $$
is equivalent to 
$$   \left(det(\frac{\partial ^2u}{\partial z_j \partial{\bar{z_k}}}) \right)^{1/n} \geq f^{1/n}. $$
The last inequality means that
$$ (dd^c u)^n \geq f \beta^n .$$

(1  $ \Rightarrow $ 2)
Let $(\rho_\epsilon)$ be a family of regularizing kernels with supp $ \rho_\epsilon \subset B(0,\epsilon) $ and $ \int_{B(0,\epsilon)} \rho_\epsilon = 1 $, hence the sequence $ u_\epsilon = u* \rho_\epsilon $ decreases to $u$, then we see that (1) implies  $ \D u_\ep \geq (f^{1/n} )_\ep $. Since $ u_\ep$ is smooth, we use the first case and get $ (dd^c u_\ep )^n \geq ( (f^{1/n})_\ep )^n \beta^n $, hence by applying the convergence theorem of Bedford and Taylor (Theorem 7.4 in \cite{BT82}) we obtain $ (dd^c u)^n \geq f \beta^n $.

\medskip

(2 $ \Rightarrow $ 1)
Fix $ x_0 \in \Om $, and $ q $ is $ \Cc^2 $-function in a neighborhood $B$ of $ x_0$ such that  $ u \leq q $  in this neighborhood and $ u(x_0) = q(x_0) $. 
 \\
First step:  we will show that $ dd^c q_{x_0} \geq 0 $. Indeed, for every small enough ball $ B^\prime \subset B $  centered at $ x_0$, we have

$$ u(x_0) - q(x_0) \geq \frac{1}{V(B^\prime)} \int\limits_{B^\prime} ( u - q) dV ,$$ 
then we get
$$ \frac{1}{V(B^\prime)} \int\limits_{B^\prime} qdV -q(x_0) \geq \frac{1}{V(B^\prime)} \int\limits_{B^\prime} u dV - u(x_0) \geq 0 . $$
Since $ q $ is $\Cc^2$-smooth and  the radius of $ B^{\prime} $ tend to 0, it follows, form Proposition 3.2.10 in \cite{H94}, that $ \Delta q_{x_0} \geq 0$. For every positive definite Hermitian matrix $H$ with $ det H = n^{-n}$, we make linear change of complex  coordinates $T$ such that $H$ will be $I$ (the identity matrix) in the new coordinates and  $ \tilde{Q}= (\partial^2 \tilde{q} / \partial w_j \partial \bar{w_k})$ where $\tilde{q}=q\circ T^{-1}$  then
$$ \D q(x_0) =  tr(H.Q)=   tr(I.\tilde{Q}) =  \Delta \tilde{q}(y_0)   $$
Hence $ \D q(x_0) \geq 0$ for every $ H \in H_n^+ $ then $ dd^c q_{x_0} \geq 0$. 

\smallskip
\noindent
Second step: we claim that $ (dd^c q)_{x_0}^n \geq f(x_0) \beta^n $. Suppose that there exists a point 
$ x_0 \in \Omega $ and a $\Cc^2 $-function $q$ which satisfies $ u \leq q $ in a neighborhood of $x_0$ and 
$ u(x_0)=q(x_0) $ such that  $ (dd^c q)_{x_0}^n < f(x_0) \beta^n $. we put
$$ q^\ep (x) = q(x) + \ep \left(\| x-x_0 \|^2  - \frac{r^2}{2}\right) $$
 for  $ 0< \ep \ll 1 $ small enough, we see that 
 $$ 0< (dd^c q^\ep )_{x_0}^n < f(x_0) \beta^n . $$
Since $f$ is lower semi-continuous on $\Omb$, there exists  $ r>0 $ such that 
\begin{center}
  $ (dd^c q^\ep )_x^n \leq f(x) \beta^n $  ;  $ x \in B(x_0,r) $.
\end{center}

\noindent
Then $ (dd^c q^\ep )^n \leq f \beta^n \leq (dd^c u)^n $ in  $ B(x_0, r)$ and  
$ q^\ep = q + \ep \frac{r^2}{2}  \geq q \geq u $ on  $ \partial B (x_0 , r) $ , hence $ q^\ep \geq u $ on $ B (x_0,r) $ by the comparison  principle. But $ q^\ep (x_0) = q(x_0) - \ep \frac{r^2}{2} = u(x_0) -\ep \frac{r^2}{2} < u(x_0) $ contradiction.

Hence, form the first part of the proof, we get $  \D q (x_0)  \geq f^{1/n}(x_0) $ for every point  $ x_0 \in \Om $ and every $ \Cc^2 $-function  $ q $ in a neighborhood of $ x_0$ such that $ u \leq q $ in this neighborhood  and $ u(x_0) = q(x_0) $.

Assume that $ f >0 $ and  $ f \in \Cc^{\infty} (\Omb) $, then there exists  $ g \in \Cc^\infty(\Omb) $ such that  $ \D g = f^{1/n} $. Hence $ \fii = u - g$ is  $\D$-subharmonic (by Proposition 3.2.10', \cite{H94}), from which it follows $ \D \fii \geq 0$ and $ \D u \geq f^{1/n} $.

In  case $f > 0$ is merely continuous, we observe that
$$ f= sup\{ w; w \in \Cc^\infty , f\geq w > 0\} ,$$
then $ (dd^c u )^n \geq f \beta^n \geq w \beta^n $. Since  $w>0$ is smooth, we have $ \D u \geq w^{1/n} $. Therefore, we get   $ \D u \geq f^{1/n} $. 

In the general case $ 0 \leq f \in \Cc(\Omb) $, we observe that $ u^\ep (z) = u(z) + \ep |z|^2 $ satisfies 
$$ (dd^c u^\ep )^n \geq (f+ \ep^n)\beta^n, $$
then 
$$ \D u^\ep \geq (f+\ep^n )^{1/n} .$$
Letting $ \ep $ converges to $0$, we get 
$$ \D u \geq f^{1/n} \text{ for all } H \in H_n^+ \text{ and } det H = n^{-n} .$$
\end{proof}

As a consequence of Proposition \ref{equi}, we give a new description of the classical Perron-Bremermann family of subsolutions to the Dirichlet problem $ Dir(\Om,\fii,f)$.

\begin{definition}
 We denote  $\fami$  the family of subsolutions of   $ Dir(\Om,\fii,f) $, that is
\begin{center}
$ \fami= \{ v \in PSH(\Om) \cap \Cc(\Omb), v|_{\Omf} \leq \fii \;and\; \D v \geq f^{1/n} \; \forall H \in H_n^+, detH=n^{-n} \} $.\\

\end{center}

\end{definition}

\smallskip

\begin{Remark}
We observe that   $ \fami \neq \emptyset $. Indeed, let $\rho$ as in  Definition \ref{lip} and  $A,B > 0$ big enough, then $ A \rho -B \in \fami$. 
\smallskip

Furthermore, the family $\fami$ is stable  under finite maximum, that is 
 if  $u,v \in \fami $ then  $ max(u,v)\in \fami $. 

\end{Remark}

\medskip

\section{The Perron-Bremermann envelope}

Bedford and Taylor proved in \cite{BT76} that  the  unique solution to $ Dir(\Om,\fii, f)$ in a bounded strongly pseudoconvex domain with smooth boundary, is given as the envelope of Perron-Bremermann 
$$ u = \sup \{ v ; v \in \mathcal{B}(\Om, \fii,f) \}$$
where $ \mathcal{B}(\Om,\fii,f) = \{ v \in PSH(\Om) \cap \Cc(\Omb) , v|_{\Omf} \leq \fii \; and \; (dd^c v)^n \geq f \beta_n \}. $
\\
Thanks to Proposition \ref{equi}, we get the following corollary
\begin{corollary} \label{coincide}
The two families $ \fami$ and $\mathcal{B}(\Om,\fii,f)$ coincide, that is
 $\fami= \mathcal{B}(\Om,\fii,f) $.
\end{corollary}

Here we will first give an alternative description of the Perron-Bremermann envelope in a bounded SHL domain. 

More precisely, we consider the upper envelope
$$\U(z) = \sup \{ v(z); v \in \fami \}.$$

\subsection{Continuity of the upper envelope}

Following the same argument in \cite{Wal69,Bl96}, we prove the continuity of the upper envelope.

\begin{theorem}\label{con}
Let $\Om \subset \C^n$ be a bounded SHL domain, $ 0 \leq f\in \Cc(\Omb) $ and $\fii\in\Cc(\Omf)$. Then the upper envelope 
$$ \U= sup\{v;\; v \in \fami\} $$
belongs to $\fami$  and $ \U = \fii$ on $ \Omf$.
\end{theorem}

\begin{proof}
Let  $ g \in \Cc^2 (\Omb) $ be an approximation of $ \fii$ such that $ |g - \fii| < \epsilon $ on $\Omf$, for $\epsilon>0$. Let also $\rho$ the defining function as in  Definition \ref{lip} and $ A>0 $ large enough such that $ v_0:= A \rho + g - \epsilon $ belongs to $\fami$.

Thus $ v_0 \leq \U  \leq h $, where $h$ be the  harmonic extension of $\fii$ to $\Om$.
Then it follows that  $ \fii - 2\epsilon \leq g- \epsilon \leq \U \leq \fii $ on  $\Omf$, as $\epsilon$ tends to $0$, we see that $ \U = \fii $ on $\Omf$.
\smallskip

We will prove that  $ \U$ is continuous on  $ \Om$ .
Fix $\ep>0 $ and $z_0$ in a compact set $K \subset \Om$. 

\noindent
Thanks to the continuity of  $h$ and $v_0$   on $\Omb$, 
one can find  $\delta> 0 $ such that for any   $ z_1,z_2 \in\Omb $ we have
\begin{center}
$ | h(z_1)-h(z_2) | < \ep $, $ | v_0(z_1)- v_0(z_2) | < \ep $  if $ | z_1-z_2 | < \delta $.
\end{center}
Let $ a \in \C^n $ such that  $ | a | < min(\delta, dist(K,\partial \Omega)) $. Since $\U$ is the upper envelope, we can find  $ v \in\fami$ such that  $ v(z_0+ a) > \U(z_0+a)-\ep $ and we can assume that $ v_0 \leq v $.

\noindent
Hence for all $ z \in \Omb$  and $ w \in \Omf$ we get
 $$-3\ep < v_0(z)-\fii(w) < v(z)-\fii(w)< h(z)-\fii(w) < \ep,$$
this implies that 
\begin{equation}\label{v}
  | v(z)- \fii(w) | < 3 \ep  \text{   if  } | z-w | <\delta .
\end{equation}
Then for $z \in \Om $ and $ z+a \in \Omf$, we have 
$$ v(z+a) \leq \fii(z+a)< v(z)+ 3 \ep . $$
We define the following function 
\begin{center}

$ v_1(z)=\left\{ \begin{array}{ll}
                   v(z)  & ;z+a \notin \Omb \\
                    max( v(z) ,v(z+a)-3\ep) & ;z+a \in \Omb\\
                   \end{array} \right. $    

\end{center}                                      
which is well defined, plurisubharmonic on  $\Om$, continuous on $\Omb$ and $ v_1 \leq \fii$ on  $\Omf $.
Indeed, if $ z \in \Omf $, $z+a \notin \Omb $ then $ v_1(z)= v(z) \leq \fii(z)$. On the other hand, if $ z \in \Omf $ and $z+a \in \Omb $ then we have, from \ref{v}, that
$$ v(z+a) - 3 \ep < \fii(z), $$ 
so $ v_1(z)= \max( v(z), v(z+a)-3 \ep ) \leq \fii(z) $.

\noindent
Moreover, we note that  $\D( v(.+a))\geq f^{1/n}(.+a)$, hence it follows that 
$$ \D v_1 \geq min(f^{1/n} , f^{1/n}(.+a)) .$$
Let $\omega$ be the modulus of continuity of $f^{1/n} $  and define
 $$ v_2= v_1 +\omega(| a |) ( v_0 - \| v_0 \|_{L^\infty(\Omb)}). $$
We claim that $ v_2  \in \fami$. It is clear that $ v_2 \in PSH(\Om) \cap \Cc(\Omb)$ and $ v_2 \leq \fii$ on $ \Omf$. Moreover, One can point out  that 
$$\D v_2 = \D v_1 + \omega(| a | ) \D v_0 \geq f^{1/n}.$$
In fact, if  $\D v_1 = f^{1/n}(.+a) $, by suitable choice of $A$ we get  
$$ \D v_2= f^{1/n}(.+a) +\omega(| a | ) \D v_0 \geq -\omega(| a |) + \omega(| a |) \D v_0 + f^{1/n} \geq f^{1/n}. $$ 
Hence we obtain that

\begin{tabular}{ll}

 $ \U(z_0 ) $  &  $\geq v_1(z_0 ) +\omega (| a | ) v_0(z_0 ) -\omega(| a |) \| v_0 \|$  \\ 
                & $ \geq v(z_0+a) -5 \ep  \;\; \; (where \; \; \omega(| a | ) <\frac{\ep}{\| v_0 \|}) $ \\
                & $ > \U(z_0+a)- 6 \ep $.  \\
 
\end{tabular} \\
Since $|a| $ is small and the last inequality is true for every $z_0 \in K$, then $\U$ is continuous on $\Om$.
\\
As the family $\fami$ is stable  under the operation maximum, we can find a sequence   $ u_j \in \fami$ such that  $ u_j $ increases almost everwhere to $\U $, then $u_j \rightarrow \U $ in $L^1(\Om)$. Hence  $\D \U = \lim \D u_j \geq f^{1/n} $  for all $H \in H_n^+$, $ detH=n^{-n}$, this implies $ \U \in \fami$.
\end{proof}
\medskip
 
\begin{proposition}\label{solution}
Let $\Om \subset  \C^n $ be a bounded strongly hyperconvex Lipschitz domain, $ 0 \leq f \in \Cc(\Omb) $  and $ \fii \in \Cc(\Omf) $. Then the Dirichlet problem  $ Dir(\Om,\fii,f)$ has a unique solution $\U$. Moreover the solution is given by 
$$ \U = \sup \{ v ; v \in \fami \} $$
where
$$ \mathcal{V} =  \{ v \in PSH(\Om) \cap \Cc(\Omb), v|_{\Omf} \leq \fii \;and\; \D v \geq f^{1/n} \; \forall H \in H_n^+, detH=n^{-n} \} $$
and $ \D $ be the laplacian associated to a positive definite Hermitian matrix $H$ as in (\ref{laplacian}).
\end{proposition}

\begin{proof}
The uniqueness follows from the comparison principle (\cite{BT76}). On the other hand, Theorem \ref{con} implies that our domain $\Om$ is $B$-regular in the sense of Sibony (\cite{Sib87}). Therefore  existence and uniqueness of the solution follows from  Theorem 4.1 in \cite{Bl96}. The description of the solution given in the proposition follows from Corollary~\ref{coincide} and Theorem \ref{con}.
\end{proof}

\begin{Remark}\label{stab}
Let $ \fii_1 , \fii_2 \in \Cc(\partial \Om)$ and $ f_1,f_2 \in \Cc( \Omb)$, then the solutions $ \U_1 = \U(\Om, \fii_1,f_1)$ , $ \U_2 = \U(\Om,\fii_2, f_2)$ satisfy the following stability estimate
\begin{equation}\label{stab}
 \|\U_1-\U_2 \|_{L^\infty (\Omb)} \leq d^2 \| f_1 - f_2\|^{1/n}_{L^\infty (\Omb)} + \| \fii_1 - \fii_2 \|_{L^\infty (\Omf)} 
\end{equation}
where $ d:= diam(\Om)$.
Indeed, fix $ z_0 \in \Om$ and define 
$$ v_1(z) = \| f_1 -f_2 \|^{1/n}_{L^{\infty}(\Omb)} ( |z-z_0|^2-d^2) + \U_2(z) $$
and
$$ v_2(z)= \U_1(z) + \| \fii_1- \fii_2\|_{L^{\infty}(\Omf)}.$$
It is clear that $ v_1 , v_2 \in PSH(\Om) \cap \Cc(\Omb)$. Hence, by the comparison principle, we get 
$ v_1 \leq v_2 $ on $ \Omb$. Then we conclude that
$$ \U_2 - \U_1 \leq d^2 \| f_1 -f_2 \|^{1/n}_{L^{\infty}(\Omb)} + \| \fii_1 - \fii_2\|_{L^{\infty}(\Omf)}$$
Reversing the roles of $\U_1$ and $\U_2$, we get the inequality (\ref{stab}).

We will need in  Section \ref{section Lp} an estimate, proved by Blocki in \cite{Bl93},  for the $L^n-L^1$ stability of solutions to the Dirichlet problem $Dir(\Om,\fii,f)$
\begin{equation}\label{Ln-L1}
\|\U_1-\U_2 \|_{L^n (\Om)} \leq \lambda(\Om) \| \fii_1 - \fii_2 \|_{L^\infty (\Omf)} + \frac{r^2}{4}  \| f_1 - f_2\|^{1/n}_{L^1 (\Om)} 
\end{equation}
where $ r = min \{ r' >0 : \Om \subset B(z_0,r')\; for \; some \; z_0 \in \C^n \} $. 
\end{Remark}

\section{The modulus of continuity of  Perron-Bremermann envelope}

Recall that a real function $\omega$ on $[0,l]$, $0<l<\infty$, is called a modulus of continuity if $\omega$ is continuous, subadditive, nondecreasing and $\omega(0)=0$.\\
In general, $\omega$ fails to be concave, we denote $\bar{\omega}$ to be the minimal concave majorant of $\omega$. 
The following  property of the minimal concave majorant $\bar{\omega}$ is well known (see \cite{Kor82} and \cite{Ch14}).

\begin{lemma}\label{mod}
Let $\omega$ be a modulus of continuity on $[0,l]$ and  $\bar{\omega}$  be the minimal concave majorant of $\omega$. Then $ \omega(\eta t)< \bar{\omega}(\eta t) < (1+\eta) \omega(t)$ for any $t>0$ and $\eta >0$.
\end{lemma}

\subsection{Modulus of continuity of the solution}
Now, we will start the first step to establish an estimate for the modulus  of continuity of the solution to $ Dir(\Omega, \fii, f)$. For this reason, it is natural to require the relation between the modulus of continuity of $\U$ and the modulus of continuity of sub-barrier and super-barrier. Thus, we present the following proposition

\begin{proposition}\label{mix}
Let $ \Om  \subset \C^n $ be a bounded SHL domain,  $\fii \in \Cc(\Omf)$ and $ 0 \leq f \in \Cc(\Omb) $. Suppose that there exist $ v \in \fami $ and $ w \in SH(\Om) \cap \Cc(\Omb)$ such that  $ v= \fii = -w $ on $\Omf$, then there is a constant $C>0$ depends on $diam(\Om)$ such that the modulus of continuity of $\U$ satisfies 
$$ \omega_\U(t) \leq C \max \{ \omega_v (t), \omega_w(t) , \omega_{f^{1/n}}(t) \}.$$
\end{proposition}

\begin{proof}
Let us put $ g(t) := max(\omega_v(t), \omega_w(t) ,\omega_{f^{1/n}}(t))$ and $ d:= diam(\Om)$. As $ v=\fii=-w$ on $\Omf$ we have for all $ z \in \Omb$ and $ \xi \in \Omf$
$$ -g( | z-\xi |) \leq v(z) - \fii(\xi) \leq \U(z) - \fii(\xi) \leq -w(z) - \fii(\xi) \leq g(| z-\xi |). $$
Hence we get
\begin{equation}\label{eq}
| \U(z) - \U(\xi) | \leq g(| z-\xi |) ; \forall z \in \Omb , \forall \xi \in \Omf .
\end{equation}
Fix a point $ z_0 \in \Om$, for any small vector $ \tau \in \C^n $, we set $ \Omega_{-\tau} := \{ z-\tau ; z \in \Omega \} $ and define in $ \Omega \cap \Omega_{-\tau } $ the function
 $$v_1(z)= \U(z+\tau)+ g(|\tau |) | z-z_0 |^2 - d^2 g(| \tau |) -  g(| \tau |) $$ 
which is well defined psh function in $\Omega \cap \Omega_{-\tau } $  and continuous on $ \Omb \cap \bar{\Omega}_{-\tau}$. 
By (\ref{eq}),  if $ z \in \Omb \cap \Omf_{-\tau}$ we can see that
\begin{equation}\label{eq1}
v_1(z) - \U(z) \leq g(| \tau |)   + g(| \tau |) | z -z_0|^2 - d^2 g(| \tau |) - g(| \tau |) \leq 0.
\end{equation}
Moreover, we assert that $\D v_1 \geq f^{1/n} $ in $ \Omega \cap \Omega_{-\tau } $ for all $ H \in H_n^+, det H= n^{-n}$. Indeed, we have

\begin{center}
$ \begin{array}{ll}
 \D v_1(z)  & \geq f^{1/n}(z+\tau) +  g(| \tau |) \D(| z-z_0 |^2) \\
   & \geq f^{1/n}(z+\tau) +  g(| \tau | ) \\
   & \geq f^{1/n}(z+\tau) + | f^{1/n}(z+\tau) - f^{1/n}(z) |  \\
   & \geq f^{1/n}(z) \\
\end{array}$
\end{center} 
for all $ H \in H_n^+ $ and $ det H= n^{-n} $.
\\
Hence, by the last properties of $ v_1$ , we find that 
\begin{center}
$ V_{\tau}(z)= \left\{ \begin{array}{ll}
                   \U(z) & ; z \in \Omb \setminus  \Om_{-\tau} \\
                   max(\U(z),v_1(z)) & ; z \in \Omb \cap \Om_{-\tau}   \\ 
\end{array} \right. $ 
\end{center}
is well defined function and belongs to $PSH(\Om) \cap \Cc(\Omb)$.
It is clear that  $\D V_\tau \geq f^{1/n} $ for all $ H \in H_n^+, det H= n^{-n}$. 
We claim that $ V_\tau = \fii $ on $\Omf$. If $ z \in \Omf \setminus \Om_{-\tau} $ then $ V_\tau(z) = \U(z) = \fii(z)$. On the other hand $ z \in \Omf \cap \Omega_{-\tau}$, by(\ref{eq1}) we get 
$ V_\tau (z) = \max(\U(z) , v_1(z)) = \U(z) = \fii(z) $.
Consequently $ V_\tau \in \fami$ and this implies that
$$ V_\tau (z) \leq \U(z) ; \forall z \in \Omb.$$
Then we have for all $ z \in \Omb \cap \Om_{-\tau} $
$$ \U(z+\tau)+ g(|\tau |) | z -z_0|^2 - d^2 g(| \tau |) -  g(| \tau |) \leq \U(z) .$$
Hence,
$$\U(z+\tau) - \U(z) \leq (d^2 + 1) g(| \tau |) -  g(| \tau |) . | z -z_0|^2 \leq C g(| \tau |) .$$
Reversing the roles of  $ z+\tau $ and $ z $, we get 
 $$ | \U(z+\tau) - \U(z) | \leq C g(| \tau | ) ; \forall z, z+\tau \in \Omb .$$
Thus, we finally get
 $$ \omega_\U (|\tau|) \leq C \max(\omega_v(|\tau|), \omega_w(|\tau|) ,\omega_{f^{1/n}}(|\tau|)).  $$
\end{proof}

\begin{Remark}
Let $ H_\fii$ be the harmonic extension of $\fii$ in a bounded SHL domain $\Om$, we can replace  $ w$ in the last proposition by $H_\fii$. It is known in the classical harmonic analysis (see \cite{Ai10}) that the harmonic extension $H_\fii$ has not, in general,  the same modulus of  continuity of $\fii$.
\\
Let us define, for small positive $t$, the modulus of continuity 
$$ \psi_{\alpha,\beta}(t) = (-log(t))^{-\alpha} t^{\beta} $$
with $ \alpha \geq 0$ and $ 0 \leq \beta < 1$. It is clear that $ \psi_{\alpha,0}$ is weaker than the H\"older continuity and $ \psi_{0,\beta} $ is the  H\"older continuity. It was shown in \cite{Ai02} that $\omega_{H_\fii}(t) \leq c \psi_{0,\beta}(t)$ for some $ c>0$ if $ \omega_\fii(t) \leq c_1 \psi_{0,\beta}(t) $ for $ \beta < \beta_0$ where $ \beta_0 <1$ depending only on $ n $ and the Lipschitz constant of the defining function $\rho$. Moreover, a similar result was proved in \cite{Ai10} for the modulus of continuity $\psi_{\alpha,0}(t) $. However, the same argument of Aikawa gives that $\omega_{H_\fii}(t) \leq c \psi_{\alpha,\beta}(t)$ for some $ c>0$ if $ \omega_\fii(t) \leq c_1 \psi_{\alpha,\beta}(t) $ for $ \alpha \geq 0$ and $ 0 \leq \beta < \beta_0 <1 $.

Hence, this leads us to the conclusion that if there exists a barrier $ v$ to the Dirichlet problem such that 
$v=\fii$ on $\Omf$ and $ \omega_v(t) \leq \lambda \psi_{\alpha,\beta}(t)$ with $ \alpha, \beta$ as above, then the last proposition gives
$$ \omega_\U \leq \lambda_1 \max \{ \psi_{\alpha,\beta} (t) , \omega_{f^{1/n}}(t) \}, $$
where $\lambda_1 >0$ depending on $\lambda$ and $ diam(\Om)$.
\end{Remark}

\subsection{Construction of barriers}
In this subsection, we will construct a subsolution to Dirichlet problem with the boundary value $\fii$ and estimate its modulus of continuity.
\begin{proposition}\label{subs}
Let $ \Om  \subset \C^n $ be a bounded SHL domain, assume  that  $\fii \in \Cc(\Omf )$ and $ 0 \leq f \in \Cc(\Omb) $. Then there exists a subsolution $ v \in \fami$ such that $ v = \fii $ on $ \Omf$ and  the modulus of continuity of $v$ satisfies the following inequality  
  $$ \omega_v(t) \leq  \lambda (1+ \|f \|^{1/n}_{L^\infty(\Omb)})  \max\{ \omega_\fii(t^{1/2}) , t^{1/2} \}  $$
where  $\lambda  > 0$ depends on $\Om$.
\end{proposition}
\noindent
Observe that we do not assume any smoothness on $ \Omf$.
\begin{proof}
First of all, let us fix  $ \xi  \in  \Omf  $ , we claim that there exists  $ v_\xi \in \fami  $  such that $ v_\xi(\xi)=\fii(\xi) $. 
It is sufficient to prove that there exists a  constant $ C>0$  depending  on $\Om $  such that for every point $ \xi \in \Omf $  and $ \fii \in \Cc(\Omf) $, there is a function $ h_{\xi} \in PSH(\Om) \cap \Cc(\Omb) $ such that\\
1) $ h_\xi(z) \leq \fii(z)  , \forall z \in \Omf $ \\
2) $h_\xi(\xi)= \fii(\xi) $\\
3) $ \omega_{h_\xi } (t) \leq C \omega_{\fii} (t^{1/2})$.\\
Assume  this is true, we fix $z_0 \in \Om$ and choose $K_1:=\sup_{\Omb} f^{1/n}  \geq 0 $, hence 
$$ \D(K_1 | z- z_0 |^2 ) = K_1 \D | z - z_0 |^2 \geq f^{1/n},  \; \; \forall H \in H_n^+ , det H=n^{-n} , $$
 we also put  $ K_2=K_1 | \xi -z_0 |^2$.  Then for the continuous function 
  $$ \tilde{\fii}(z):= \fii(z) - K_1 | z - z_0 |^2 + K_2 , $$
 we have  $ h_\xi$ such that 1),2)and  3)  hold.

\noindent
Then the desired function   $ v_\xi \in \fami  $ is given by  
$$ v_\xi(z) = h_\xi(z) + K_1 | z - z_0 |^2 - K_2 $$ 
Because,  $ h_\xi(z) \leq \tilde{\fii}(z)= \fii(z) -K_1 | z - z_0 |^2 + K_2 $ on  $ \Omf $,
 so  $ v_\xi(z) \leq \fii $ on  $ \Omf$ and  $ v_\xi(\xi)= \fii(\xi) $.
 \\
Moreover, it is clear that  
$$ \D v_\xi = \D h_\xi + K_1 \D( | z - z_0 |^2) \geq f^{1/n}, \; \; \forall H \in H_n^+ , det H=n^{-n}. $$ 
Furthermore, using the hypothesis of $h_\xi$, we can control the modulus of continuity of $v_\xi$ 
 \begin{center}
 $\begin{array}{ll}
 \omega_{v_\xi} (t)= \sup \limits_{|z-y| \leq t } |v_\xi (z)-v_\xi (y)|  & \leq \omega_{h_\xi} (t) + K_1 \omega_{|z-z_0|^2} (t)\\
   & \leq C \omega_{\tilde{\fii}}(t^{1/2}) + 4 d^{3/2} K_1 t^{1/2} \\
  & \leq  C \omega_\fii (t^{1/2}) +2d K_1(C+2d^{1/2}) t^{1/2} \\
  & \leq (C+2d^{1/2})(1+2dK_1)\max \{\omega_\fii(t^{1/2}), t^{1/2} \} .\\
\end{array} $
 \end{center}
Hence, we conclude that 
 $$  \omega_{v_\xi} (t) \leq \lambda (1+   K_1) \max \{ \omega_\fii(t^{1/2}) , t^{1/2} \}$$
 where  $\lambda := (C+2d^{1/2})(1+2d)$ is a positive constant depending on $ \Om$. 
\\
Now we will construct   $h_\xi \in PSH(\Om) \cap \Cc(\Omb)$ which satisfies the three conditions above.\\
Let  $B>0$ large enough such that the function 
$$ g(z)=  B \rho (z) - | z-\xi |^2 $$
 is psh in $\Om$.
Let $ \bar{\omega}_\fii $  be the minimal concave majorant of $ \omega_{\fii}$ and  define
$$ \chi (x) = - \bar{\omega}_\fii ((-x)^{1/2}) $$
which is convex nondecreasing function on $[-d^2,0]$.
 Now fix $ r>0$ so small that $ |g(z)| \leq d^2$ in $B(\xi,r)\cap  \Om$ and define for $ z \in B(\xi,r) \cap \Omb $ the function
$$ h(z) = \chi \circ g(z) + \fii(\xi). $$
It is clear that $h$ is continuous psh function on $ B(\xi,r) \cap \Om $ and we see that $ h(z) \leq \fii (z) $ if $ z \in B(\xi ,r) \cap \Omf $ and $ h(\xi ) = \fii (\xi) $.
Moreover by the subadditivity of $ \bar{\omega}_{\fii} $ and Lemma \ref{mod} we have
\begin{center}
$  \begin{array}{ll}
       \omega_h(t)& = \sup\limits_{ | z-y| \leq t } | h(z) - h(y) | \\
       & \leq \sup\limits_{ | z-y|  \leq t } \bar{\omega}_{\fii} \left[  \left| | z- \xi |^2 - | y- 
       \xi|^2 -  B(\rho(z) - \rho(y) ) \right| ^{1/2} \right] \\
       & \leq  \sup\limits_{ | z-y| \leq t } \bar{\omega}_{\fii} \left[ \left( |z-y| (2d+ B_1 ) 
       \right)^{1/2}  \right] \\
       & \leq  C . \omega_\fii (t^{1/2})  \\
  \end{array}
$
\end{center}
where $ C:= 1+ (2d+B_1)^{1/2}$ depends on $ \Om$.
\\
Recall that $ \xi \in \Omf$ and 
fix $ 0 < r_1 < r$ and $ \gamma_1 \geq d/r_1 $ such that 
$$ - \gamma_1 \bar{\omega}_\fii \left[ ( | z- \xi |^2 - B \rho (z) )^{1/2} \right] \leq \inf_{\Omf} \fii - \sup_{\Omf} \fii, $$
 for $ z \in \Omf \cap \partial B(\xi , r_1) $. Set $ \gamma_2 = \inf\limits_{\Omf} \fii  $, then it follows that 
 $$ \gamma_1 (h(z)-\fii(\xi)) + \fii(\xi) \leq \gamma_2 \text{  for  }  z \in \partial B(\xi, r_1) \cap \Omb .$$
Now let us put 
 $$ h_\xi(z)=\left\{ \begin{array}{ll}
                   max[\gamma_1 ( h(z)-\fii(\xi)) + \fii(\xi), \gamma_2 ] & ;z \in \Omb \cap   
                   B(\xi,r_1)  \\
                    \gamma_2 & ; z  \in \Omb \setminus B(\xi,r_1)  \\
                   \end{array} \right. $$  
which is well defined plurisubharmonic function on $\Om$, continuous on $\Omb$  and satisfies that $ h_\xi (z) \leq \fii(z)  $ for all $ z \in \Omf$. Indeed, on $\Omf \cap B(\xi,r_1)$ we have 
$$ \gamma_1 (h(z) - \fii(\xi) ) + \fii(\xi) = - \gamma_1 \bar{\omega}_\fii ( | z- \xi | ) + \fii(\xi) \leq - \bar{\omega}_\fii ( | z- \xi | ) + \fii(\xi) \leq \fii(z). $$ 
Hence it is clear that $ h_\xi$ satisfies the three conditions above.
\\
We have just proved that for each $\xi \in \Omf $, there is a function 
$$ v_\xi \in \fami, \; v_\xi(\xi)=\fii(\xi)  \text{ and } \omega_{ v_\xi} (t) \leq \lambda (1+   K_1) \max \{ \omega_\fii(t^{1/2}) , t^{1/2 } \}.  $$
Let us set 
  $$ v(z) = sup \left\{ v_\xi(z) ; \xi \in \Omf \right\}. $$
We can note $ 0 \leq \omega_v (t) \leq \lambda (1+   K_1) \max \{ \omega_\fii(t^{1/2}), t^{1/2} \}  $, then $ \omega_v(t) $ converges to zero when $t$ converges to zero. Consequently, we get $v \in \Cc(\Omb) $ and $ v=v^* \in PSH(\Om) $. Thanks to Choquet lemma, we can choose a nondecreasing sequence $(v_j)$, where $v_j \in \fami$, converging to $ v$ almost everywhere. This implies that 
$$\D v = \lim\limits_{j \to \infty} \D v_j \geq f^{1/n} , \forall H \in H_n^+ , detH = n^{-n}. $$
It is clear that   $ v(\xi) = \fii(\xi) $ for any $  \xi \in \Omf $. Finally, we get 
  $ v \in \fami $ ,  $ v = \fii $ on  $\Omf$ and 
  $ \omega_v (t) \leq \lambda (1+K_1)  \max \{ \omega_\fii(t^{1/2}) , t^{1/2} \}  $.
\end{proof}

\begin{Remark} If we assume that $\Omega$  has a smooth boundary and $\fii$ is  $C^{1,1}$-smooth, then it is possible to  construct a Lipschitz barrier $v$ to the Dirichlet problem $ Dir(\Om, \fii,f)$ (see Theorem 6.2 in \cite{BT76}).

\end{Remark}
\begin{corollary}\label{sups}
Under the same assumption of Proposition \ref{subs}. There exists a plurisuperharmonic function  $ \tilde{v} \in   \Cc(\Omb) $ such that $ \tilde{v} = \fii $ on $ \Omf $ and 
$$\omega_{\tilde{v}} (t) \leq \lambda (1+ \|f \|^{1/n}_{L^\infty(\Omb)})  \max \{ \omega_\fii (t^{1/2} ), t^{1/2} \}  ,$$
where $\lambda >0$ depends on $\Om$.
\end{corollary} 

\begin{proof}
We can do the same construction as in the proof of Proposition \ref{subs} for the function $ \fii_1=- \fii \in \Cc(\Omf) $, then  we get $  v_1 \in \mathcal{V}(\Om,\fii_1,f)$  such that $ v_1 = \fii_1 $ on  $\Omf$ and $\omega_{v_1} (t) \leq (1+ \|f \|^{1/n}_{L^\infty(\Omb)})  \max \{ \omega_\fii (t^{1/2} ), t^{1/2} \}   $. Hence, we set 
$ \tilde{v} = - v_1 $ which is a plurisuperharmonic function on $\Om$, continuous on $\Omb$ and satisfies $ \tilde{v} = \fii $ on $ \Omf $ and
 $\omega_{\tilde{v}} (t) \leq \lambda  (1+ \|f \|^{1/n}_{L^\infty(\Omb)}) \max \{ \omega_\fii (t^{1/2} ), t^{1/2} \}   $.
\end{proof}

\subsection{Proof of Theorem A}
Thanks to  Proposition \ref{subs}, we obtain a subsolution    $ v \in \fami  $  , $ v = \fii $ on $\Omf$ and
 $$ \omega_v (t) \leq \lambda (1+\| f\|^{1/n}_{L^{\infty}(\Omb)}) \max \{ \omega_\fii (t^{1/2} ), t^{1/2} \} . $$
Observing  Corollary \ref{sups}, we get $ w \in PSH(\Omega) \cap \Cc(\Omb) $ such that  $ w = -\fii $ on $\Omf $ and 
$$ \omega_w(t) \leq \lambda (1+\| f\|^{1/n}_{L^{\infty}(\Omb)} ) \max \{ \omega_\fii (t^{1/2} ), t^{1/2} \}   $$
where $\lambda >0$ constant.
Applying the Proposition \ref{mix}  we get the wanted result, that is
$$ \omega_\U (t) \leq \eta (1+\| f\|^{1/n}_{L^{\infty}(\Omb)} )  \max \{\omega_\fii(t^{1/2}), \omega_{f^{1/n}}(t) , t^{1/2} \}  $$
where $\eta >0$ depends on $\Omega$.

\begin{corollary}\label{Holder}
Let $\Om$ be a bounded SHL domain in $\C^n$. Let $ \fii \in  \Cc^{0,\alpha}(\Omf) $ and $0 \leq f^{1/n} \in \Cc^{0,\beta}(\Omb) $, $ 0< \alpha, \beta \leq 1 $. Then 
the solution $\U$ to the  Dirichlet problem $Dir(\Om,\fii,f)$ belongs to $ \Cc^{0,\gamma}(\Omb) $ for $ \gamma= \min (\beta, \alpha/2)$. 
\end{corollary}

\noindent
The following example illustrates that the estimate of $\omega_\U$ in Theorem A is optimal.

\begin{example}
Let $ \psi$ be a concave modulus of continuity on $[0,1]$ and
 $$ \fii(z) = -\psi[\sqrt{(1+Rez_1)/2}], \text{ for  }  z=(z_1,z_2,..., z_n) \in \partial\B \subset \C^n. $$
It is easy to show that $ \fii \in \Cc(\partial \B)$ with modulus of continuity 
$$ \omega_{\fii}(t) \leq C \psi(t)$$
for some $C>0$.
\\
Let $ v(z)= -(1+Rez_1)/2 \in PSH(\B) \cap \Cc(\bar{ \B})$ and
 $ \chi(\lambda)= - \psi(\sqrt{-\lambda})$ is convex increasing function on $ [-1,0]$.  Hence we get that
 $$ u(z)= \chi \circ v(z) \in PSH(\B)\cap \Cc(\bar{ \B})$$
 and satisfies $ (dd^c u)^n =0$ in $ \B$ and $ u= \fii $ on $ \partial \B$.
The modulus of continuity of $ \U$, $ \omega_\U(t)$, has the estimate
$$ C_1 \psi(t^{1/2}) \leq \omega_\U(t) \leq C_2  \psi(t^{1/2})$$
for $ C_1,C_2>0$.
\\
Indeed, let $ z_0=(-1,0,...,0)$ and $ z=(z_1,0,...,0) \in \B $ where $ z_1= -1+2t$ and $0 \leq t\leq 1$. Hence, by Lemma \ref{mod}, we see that 
$$ \psi(t^{1/2})= \psi [ \sqrt{|z-z_0|/2} ] = \psi[\sqrt{(1+Rez_1)/2}] = |\U(z) - \U(z_0)| \leq  \omega_\U( 2t) \leq 3 \omega_\U(t).  $$ 
\end{example}
\begin{definition}
Let $ \psi $ be a modulus of continuity,  $ E \subset \C^n $ be a bounded set and $ g \in \Cc\cap L^\infty(E)$. We define the norm of $ g$ with respect to $ \psi$ ( $\psi$-norm) as follows
$$ \| g \|_{\psi} := \sup\limits_{z \in E} |g(z)| + \sup\limits_{ z \neq y \in E } \frac{|g(z)-g(y)|}{\psi(|z-y|)}$$
\end{definition}

\begin{proposition}
Let $\Omega \subset \C^n$ be a bounded SHL domain, $ \fii \in \Cc(\Omf)$ with modulus of continuity $ \psi_1$ and $ f^{1/n} \in \Cc(\Omb)$ with modulus of continuity $ \psi_2$. Then there exists a constant $C>0$ depending on $\Om$ such that 
$$ \| \U \|_{\psi} \leq C (1+ \| f\|^{1/n}_{L^{\infty}(\Omb)}) \max \{ \| \fii \|_{\psi_1} , \| f^{1/n} \|_{\psi_2} \}$$
where $ \psi(t) = \max\{\psi_1(t^{1/2}), \psi_2(t)\}$.
\end{proposition}

\begin{proof}
By hypothesis, we see that $ \|\fii \|_{\psi_1} < \infty $ and  $ \| f^{1/n} \|_{\psi_2} < \infty $.
Let $ z \neq y \in \Omb$, by Theorem A, we get 

\begin{tabular}{ll}
$|\U(z)-\U(y)|$ &  $ \leq \eta (1+ \| f\|^{1/n}_{L^{\infty}(\Omb)}) \max \{ \omega_\fii (|z-y|   
                   ^{1/2} ) , \omega_{f^{1/n}}( |z-y|)\} $ \\
                & $ \leq \eta    (1+ \| f\|^{1/n}_{L^{\infty}(\Omb)}) \max \{ \| \fii \|_{\psi_1}  
                  \psi_1(|z-y|^{1/2}), \| f^{1/n} \|_{\psi_2} \psi_2(|z-y|) \}$ \\
                & $ \leq   \eta    (1+ \| f\|^{1/n}_{L^{\infty}(\Omb)}) \max \{\| \fii \|
                _{\psi_1}, \| f^{1/n} \|_{\psi_2} \} \psi(|z-y|)$ \\
\end{tabular}

\noindent
where $ \psi(|z-y|)= \max \{\psi_1(|z-y|^{1/2}), \psi_2(|z-y|) \} $.
\\
Hence we have 
$$ \sup\limits_{z\neq y \in \Omb} \frac{|\U(z)-\U(y)|}{\psi(|z-y|)} \leq  \eta    (1+ \| f\|^{1/n}_{L^{\infty}(\Omb)}) \max \{\| \fii \|_{\psi_1}, \| f^{1/n} \|_{\psi_2} \}$$
where $ \eta \geq d^2 +1$ and $ d= diam(\Om)$ (see Proposition \ref{mix}).
From Remark \ref{stab}, we note that 
$$\| \U\|_{L^{\infty} (\Omb) } \leq d^2 \| f\|^{1/n}_{L^\infty(\Omb)} + \| \fii \| _{L^\infty(\Omf)}  \leq  \eta \max \{ \| \fii \|_{\psi_1}, \| f^{1/n} \|_{\psi_2} \}$$
Then we can conclude that 
$$  \| \U \|_{\psi} \leq 2 \eta (1+ \| f\|^{1/n}_{L^{\infty}(\Omb)}) \max \{ \| \fii \|_{\psi_1} , \| f^{1/n} \|_{\psi_2} \}.$$

\end{proof}

Finally, it is natural to try to relate the modulus of continuity of $ \U:= \U(\Om,\fii,f)$ to the modulus of continuity of $\U_0:= \U(\Om,\fii,0)$ the solution to Bremermann problem in a bounded SHL domain.

\begin{proposition}
Let $\Om$ be a bounded SHL domain in $\C^n$, $ f \in \Cc(\Omb)$ and $ \fii \in \Cc(\Omf)$. Then there exists a positive constant $C = C(\Om)$ such that 
$$ \omega_\U(t) \leq C (1+ \|f\|^{1/n}_{L^\infty(\Omb)} ) \max \{ \omega_{\U_0}(t), \omega_{f^{1/n}}(t) \}. $$
\end{proposition}

\begin{proof}
First, we search a subsolution  $ v \in \fami $ such that $ v|_{\Omf} = \fii$ and estimate its modulus of continuity. 
Since $ \Om$ is bounded  SHL  domain, there exists a  Lipschitz defining function $\rho$ on $\Omb$.
Let us  define the function
 $$ v(z)= \U_0(z) + A \rho(z) $$
where $A := \|f\|^{1/n}_{L^\infty} / c  $ and $ c>0$ as in the Definition \ref{lip}. It is clear that  $ v \in \fami$, $ v= \fii$ on $\Omf$ and   $ \omega_v(t) \leq \tilde{C} \omega_{\U_0}(t) $ where $\tilde{C} := \gamma (1+ \|f\|^{1/n}_{L^\infty(\Omb)} )$ and $\gamma \geq 1 $ depends on $ \Om$.
On the other hand, by the comparison principle we get that $ \U \leq \U_0$. Hence 
$$ v \leq \U \leq \U_0 \text{  in } \Om  \text{  and  } v=\U = \U_0 = \fii \text{  on } \Omf .$$
Thanks to Proposition \ref{mix}, there exists $ \lambda >0$ depending on $ \Om$ such that
$$ \omega_\U(t) \leq \lambda \max \{ \omega_v (t), \omega_{\U_0} (t), \omega_{f^{1/n}}(t)\}.$$
Hence, the following inequality holds for some $ C>0$ depending on $\Om$ 
$$ \omega_\U(t) \leq C (1+ \|f\|^{1/n}_{L^\infty(\Omb)} ) \max \{ \omega_{\U_0} (t),\omega_{f^{1/n}}(t) \}.$$
\end{proof}

\section{ H\"older continuous solutions for the Dirichlet problem with $L^p$ density}\label{section Lp}

In this section we will prove the existence and the H\"older continuity of the solution to Dirichlet problem $Dir(\Om,\fii,f)$ when $ f\in L^p(\Om)$, $p>1$ in a bounded SHL domain.

It is well known in   \cite{Ko98} that there exists a weak continuous solution to this  problem when $\Om$ is a bounded strongly pseudoconvex domain with smooth boundary.

The H\"older continuity of this solution was studied in \cite{GKZ08} under some additional conditions on the density and on the boundary data, that is when $f$ is bounded near the boundary and $ \fii \in \Cc^{1,1}(\Omf)$.

An essential method in this study is played by an a priori weak stability estimate of the solution which is still true when $\Om$ is a bounded SHL domain. More precisely, we have the following theorem

\begin{theorem}(\cite{GKZ08}).\label{stability Lp}
Fix $ 0 \leq f \in L^p(\Om)$, $p>1$. Let $ u, v$ be two bounded plurisubharmonic functions in $\Om$ such that $ (dd^c u)^n=f \beta^n$ in $\Om$ and let $ u \geq v$ on $\Omf$. Fix $ r \geq 1$ and $ 0 \leq \gamma < r/(nq+r)$, $1/p+1/q=1$. Then there exists a uniform constant $C= C(\gamma,n,q) >0$ such that 
$$ \sup_{\Om} (v - u) \leq C (1+ \| f\|^\tau_{L^p(\Om)}) \| (v-u)_+   \|_{L^r(\Om)}^\gamma $$
where $ \tau := \frac{1}{n}+ \frac{\gamma q}{r-\gamma(r+nq)}$ and $ (v - u)_+ := \max (v-u, 0)$.
\end{theorem}  

It was constructed in \cite{GKZ08} a Lipschitz continuous barrier to the Dirichlet problem when $ \fii \in \Cc^{1,1}(\Omf)$ and $ f $ is bounded near the boundary. Moreover, it was shown in this case that the total mass of $\Delta \U$ is finite in $\Om$. Finally, they conclude that $ \U \in \Cc^{0,\al}(\Omb)$ for any $ \alpha < 2/(nq+1)$.
However, the following theorem summarizes the work introduced in \cite{GKZ08} 

\begin{theorem}(\cite{GKZ08}).\label{Lp}
Let $ 0 \leq f \in L^p(\Om)$, for some $ p>1$ and $ \fii \in  \Cc(\Omf)$. Suppose that there exists $ v,w \in PSH(\Om) \cap \Cc^{0,\al}(\Omb)$ such that $ v \leq \U \leq -w $ on $\Omb$ and $ v =  \fii = -w $ on $\Omf$. If the total mass of $ \Delta \U $ is finite in $\Om$, then $ \U \in \Cc^{0,\alpha'}(\Omb)$ for  $ \alpha' < \min \{ \alpha, 2/(nq+1) \}$.  
\end{theorem}

Here let $\Om \subset \C^n $ be a bounded SHL domain. Using the stability theorem \ref{stability Lp} we will ensure the existence of the solution to the Dirichlet problem $Dir(\Om,\fii,f)$.

\begin{proposition}
Let $\Om \subset \C^n$ be a bounded SHL domain, $\fii \in \Cc(\Omf)$ and $ f \in L^p(\Om)$ for some $ p>1$. Then there exists a unique solution $\U$ to the Dirichlet problem $ Dir(\Om,\fii,f)$.
\end{proposition}

\begin{proof}
Let $ (f_j)$ be a sequence of smooth functions on $\Omb$ which converges to $f$ in $ L^p(\Om)$. Thanks to Proposition \ref{solution}, there exists a unique solution $\U_j$ to $ Dir(\Om,\fii,f_j)$ that is 
$ \U_j \in PSH(\Om) \cap \Cc(\Omb)$, $ \U_j = \fii$ on $\Omf$ and $ (dd^c \U_j)^n = f_j \beta^n$ in $\Om$.
We claim that 
\begin{equation}\label{approximation}
\| \U_k - \U_j \|_{L^\infty(\Omb)} \leq A (1+ \| f_k\|^\tau_{L^p(\Om)})(1+ \| f_j\|^\tau_{L^p(\Om)}) \| f_k -f_j \|^{\gamma/n}_{L^1(\Om)}
\end{equation}
where $ 0 \leq \gamma < 1/(q+1)$ fixed, $ \tau := \frac{1}{n}+ \frac{\gamma q}{n-\gamma n(1+q)}$ and $ A= A(\gamma, n,q, diam(\Om))$. 
\\
Indeed, by the stability theorem \ref{stability Lp} and for $ r=n$, we get that 
$$
 \sup_{\Om} (\U_k - \U_j) \leq C (1+ \| f_j\|^\tau_{L^p(\Om)}) \|( \U_k - \U_j)_+ \|_{L^n(\Om)}^\gamma \leq C (1+ \| f_j\|^\tau_{L^p(\Om)})  \| \U_k - \U_j \|_{L^n(\Om)}^\gamma
$$
where $ 0 \leq \gamma < 1/(q+1)$ fixed and $ C = C (\gamma,n,q) >0$.
\\
Hence by the $L^n-L^1$ stability theorem in \cite{Bl93} (see here Remark \ref{stab}), we get 
$$
\| \U_k - \U_j \|_{L^n(\Om)} \leq \tilde{C} \|f_k - f_j \|^{1/n}_{L^1(\Om)},
$$
where $ \tilde{C}$ depends on $ diam(\Om)$.
\\
Then, by combining the last two inequalities, we get
$$
\sup_{\Om} (\U_k - \U_j) \leq C \tilde{C}^\gamma (1+ \| f_j\|^\tau_{L^p(\Om)}) \|f_k - f_j \|^{\gamma/n}_{L^1(\Om)}
$$
Reversing the roles of $ \U_j$ and $\U_k$  we see that
$$
\sup_{\Om} (\U_j - \U_k) \leq C \tilde{C}^\gamma (1+ \| f_k\|^\tau_{L^p(\Om)}) \|f_k - f_j \|^{\gamma/n}_{L^1(\Om)}
$$
Hence we conclude that
$$
\| \U_k - \U_j \|_{L^\infty(\Om)} \leq C \tilde{C}^\gamma  (1+ \| f_k\|^\tau_{L^p(\Om)})(1+ \| f_j\|^\tau_{L^p(\Om)}) \| f_k -f_j \|^{\gamma/n}_{L^1(\Om)}
$$
Since $ \U_k=\U_j=\fii$ on $\Omf$, we get the inequality (\ref{approximation}). 
\\
Since $ f_j$ conveges to $f$ in $L^p(\Om)$, there is a uniform constant $ B>0$ such that 
$$
\| \U_k - \U_j \|_{L^\infty(\Omb)} \leq B
$$
This implies that the sequence $ \U_j$ converges uniformly in $\Omb$. Let us put $ \U = \lim \U_j$, it is clear that $ \U \in PSH(\Om) \cap \Cc(\Omb)$, $ \U = \fii$ on $\Omf$. Moreover, $ (dd^c \U_j)^n $ converges to $ (dd^c \U)^n $ in the sense of currents, then $(dd^c \U)^n =f \beta^n$ in $\Omega$. The uniqueness of the solution comes from the comparison principle (see \cite{BT76}).
\end{proof}

Our next step is to construct  H\"older continuous sub-barrier and super-barrier to the Dirichlet problem when $ f \in L^p(\Om)$ for some $ p>1$ and $\fii \in \Cc^{0,1}(\Omf)$. 

\begin{proposition}\label{barrier2}
Let $\fii \in \Cc^{0,1}(\Omf)$ and $ 0 \leq f \in L^p(\Om)$, for some $ p  >1$. Then there exist  $v,w \in PSH(\Om)\cap \Cc^{0,\al}(\Omb)$ where $ \alpha < 1/(nq+1)$ such that $ v=\fii=-w $ on $\Omf$ and $v \leq \U \leq -w$ on $\Om$.
\end{proposition}

\begin{proof}
Fix a large ball $ B \subset \C^n$ so that $ \Om \Subset B \subset \C^n$. Let $ \tilde{f} $ be a trivial extension of $f$ to $B$. Since $ \tilde{f}\in L^p(\Om)$ is bounded near $ \partial B$, the solution $h_1$ to $ Dir(B,0,\tilde{f})$ is H\"older continuous  on $ \bar{B}$ with exponent $ \alpha_1 < 2/(nq+1)$ (see \cite{GKZ08}). Now let $h_2$ denote the solution to the Dirichlet problem in $\Om$ with boundary values $\fii-h_1$ and the zero density.
Thanks to Theorem A, we see that $ h_2 \in \Cc^{0,\alpha_2}(\Omb)$ where $ \alpha_2= \alpha_1/2$.
Therefore, the required barrier will be $ v =  h_1 + h_2$. It is clear that $ v \in PSH( \Om) \cap \Cc(\Omb)$, $ v|_{\Omf} = \fii$ and $ (dd^c v)^n \geq f \beta^n$ in the weak sense in $\Om$. Hence, by the comparison principle we get that $ v \leq \U  $ in $\Om$ and $ v=\U=\fii$ on $\Omf$. Moreover we have $ v \in \Cc^{0,\al}(\Omb)$ for any $ \alpha < 1/(nq+1)$.
\\
Finally, it is enough to set $ w= \U(\Om,-\fii,0)$ to obtain a super-barrier to the Dirichlet problem $ Dir(\Om,\fii,f)$. We note that $ w \in PSH(\Om)\cap \Cc(\Omb)$, $ -w= \fii$ on $\Omf$ and $ \U \leq -w$ on $\Omb$. Furthermore, by Theorem A, $w \in \Cc^{0,1/2}(\Omb)$ and then $ w \in  \Cc^{0,\alpha}(\Omb)$  for any $ \alpha < 1/(nq+1)$.
\end{proof}

When $ f \in L^p(\Om)$ for $ p \geq 2$, we are able to find a H\"older continuous barrier to the Dirichlet problem with more better H\"older exponent. The following theorem was proved in \cite{Ch14}  for the complex Hessian equation and it is enough here to put $m=n$ for the complex Monge-Amp\`ere equation.

\begin{theorem}(\cite{Ch14}).\label{barrier1}
Let $\fii \in \Cc^{0,1}(\Omf)$ and $ 0 \leq f \in L^p(\Om)$, $ p  \geq 2$. Then there exist $ v,w \in PSH(\Om) \cap \Cc^{0,1/2}(\Omb) $ such that $ v = \fii=-w $ on  $ \Omf$ and  $ v \leq \U \leq -w $ in $\Om$.
\end{theorem}

Now we recall the comparison principle for the total mass of laplacian of plurisubharmonic functions.

\begin{lemma}\label{mass}
Let $ u,v \in PSH(\Om) \cap \Cc(\Omb)$ such that $ v \leq u$ on $\Om$ and $u=v$ on $\Omf$. Then 
$$ \int_{\Om} dd^c u \wedge \beta^{n-1} \leq \int_{\Om} dd^c v \wedge \beta^{n-1}.$$
\end{lemma}

\begin{proof}
First assume that there exists an open set $ V \Subset \Om$ such that $ u=v$ on $\Omb \setminus V$. Let $ h \in PSH (\Om) \cap \Cc(\Omb)$ such that  $ h=0$ on $\Omf$. Then integration by parts yields 
$$ \int_{\Om} h dd^c(v- u) \wedge \beta^{n-1} = \int_{\Om} (v-u) dd^c h \wedge \beta^{n-1}.$$
Let $ V_1$ be an open set such that $ V \Subset V_1 \Subset \Om$ and define the function  $ h= \max (-1, \rho /m)$ where $ \rho $ be the defining function of $ \Om$ and $ m= |\sup_{\partial V_1} \rho |$. It is clear that $ h \in PSH(\Om) \cap \Cc(\Omb)$, $h=0$ on $ \Omf$ and $ h=-1 $ on $\bar{V_1}$.
Since $u=v$ on $\Omb \setminus V$, we get
$$  \int_{\Om}  dd^c(v- u) \wedge \beta^{n-1} =  \int_{V_1} dd^c(v- u) \wedge \beta^{n-1}. $$
\\
We note that 

\begin{tabular}{ll}

$\int_{V_1}  dd^c(v- u) \wedge \beta^{n-1}$& $=-\int_{V_1} h dd^c(v- u) \wedge \beta^{n-1} $\\
&$=- \int_{\Om} h dd^c(v- u) \wedge \beta^{n-1}$ \\
& $ =- \int_{\Om} (v- u)  dd^c h \wedge \beta^{n-1} \geq 0 . $\\

\end{tabular}
\\
Hence we obtain 
$$ 
\int_{\Om} dd^c u \wedge \beta^{n-1} \leq \int_{\Om} dd^c v \wedge \beta^{n-1}.
$$
Now if we have only $ u=v$ on $\Omf$, then we define for small $ \epsilon>0$, the function $ u_\epsilon:= \max(u-\epsilon, v)$. Then we see that $ v \leq u_\epsilon$ on $\Om$ and $ u_\epsilon =v $ near the boundary of $\Om$.
\\
Therefore, we have 
$$
 \int_{\Om} dd^c u_\epsilon \wedge \beta^{n-1} \leq \int_{\Om} dd^c v  \wedge \beta^{n-1}.
$$
We know by the convergence's theorem of Bedford and Taylor that $ dd^c u_\epsilon \beta^{n-1} \rightharpoonup dd^c u \wedge \beta^{n-1}$ when $ \epsilon \searrow 0$. Thus we have
$$
 \int_{\Om} dd^c u \wedge \beta^{n-1} \leq \int_{\Om} dd^c v  \wedge \beta^{n-1}.
$$
which proves the required inequality.
\end{proof}

\subsection{Proof Theorem B}
Let $\U_0$ the solution to the Dirichlet problem $ Dir(\Om,0,f)$.
We first claim that the total mass of $ \Delta \U_0$ is finite in $\Om$. Indeed, let $ \rho$ be the  defining function of $\Om$, then by Corollary 5.6 in \cite{Ce04} we get that
\begin{equation}\label{cegrell inequlaity}
\begin{array}{ll}
\int_{\Om} dd^c \U_0 \wedge (dd^c \rho)^{n-1} & \leq \left(\int_{\Om} (dd^c \U_0)^n \right)^{1/n} \cdot  \left( \int_{\Om} ( dd^c \rho )^n \right)^{(n-1)/n} \\
& \leq \left(\int_{\Om} f \beta^n \right)^{1/n} \cdot  \left( \int_{\Om} ( dd^c \rho )^n \right)^{(n-1)/n}. \\ 
\end{array} 
\end{equation}  
Since $\Om$ is a bounded SHL domain, there exists a constant $c>0$ such that $ dd^c \rho \geq c \beta$ in $\Omega$. Hence the inequality \ref{cegrell inequlaity}   yields 
$$ 
\begin{array}{ll}
\int_{\Om} dd^c \U_0 \wedge \beta^{n-1}  &  \leq\frac{1}{c^{n-1}}\int_{\Om} dd^c \U_0 \wedge               
 (dd^c \rho)^{n-1}  \\
& \leq \frac{1}{c^{n-1}}\left(\int_{\Om} f \beta^n \right)^{1/n} \cdot  \left( \int_{\Om} ( dd^c \rho )^n \right)^{(n-1)/n} \\ 
\end{array}
$$
\noindent
Now we note that the total mass of complex Monge-Ampere measure of $\rho$  is finite in $\Om$ by Chern-Levine-Nirenberg inequality since $\rho$ is psh and bounded in a neighborhood of $\Omb$ ( see \cite{BT76}).
Therefore, the total mass of $\Delta \U_0$ is finite in $\Om$.
\\
Let $\tilde{\fii}$ be a $\Cc^{1,1}$-extension of $\fii$ to $\Omb$ such that $ \| \tilde{\fii}\|_{\Cc{1,1}(\Omb)} \leq C \| \fii \|_{\Cc^{1,1}(\Omf)}$ for some $C>0$.
Now, let $ v= A\rho + \tilde{\fii} + \U_0$ where $ A \gg 1$ such that $ A \rho +\tilde{\fii} \in PSH(\Om)$. By the comparison principle we see that $ v \leq \U$ in $\Om$ and $ v=\U =\fii$ on $\Omf$. Since $ \rho$ is psh in a neighborhood of $ \Omb$ and $ \| \Delta \U_0 \|_\Om < +\infty$, we get that $  \| \Delta v \|_\Om < +\infty$.
Then  by Lemma \ref{mass} we have $ \| \Delta\U \| < + \infty$.

\noindent
The Proposition \ref{barrier2} gives the existence of  H\"older continuous barriers to the Dirichlet problem. Then using Theorem \ref{Lp} we obtain the final result that is when $ f \in L^p(\Om)$ for some $p>1$, we get
$\U \in PSH(\Om)\cap \Cc^{0,\al}(\Omb)$ where $ \alpha < 1/(nq+1)$.
\\
Moreover, if $ f \in L^p(\Om)$ for some $p \geq 2$,  we can get  better result. By Theorem \ref{barrier1} and Theorem \ref{Lp}, we see that  
$\U \in PSH(\Om)\cap \Cc^{0,\al}(\Omb)$ where $ \alpha < \min \{ 1/2, 2/(nq+1) \}$.

\begin{Remark}
It is shown in \cite{GKZ08} that we cannot expect a better H\"older exponent than $ 2/nq$ (see also \cite{Pl05}).
\end{Remark}

\bigskip

\noindent 
Mohamad Charabati \\
Institut de Math\'ematiques de Toulouse \\
Universit\'e Paul Sabatier \\
118 route de Narbonne \\
31602 Toulouse Cedex 09 (France). \\
e-mail: {\tt mohamad.charabati@math.univ-toulouse.fr }

\end{document}